\documentclass[11pt]{amsart}
\usepackage[english]{babel}
\usepackage{amssymb,amsmath,xcolor,mathrsfs}

\usepackage{graphicx,pstricks,pst-plot}
\usepackage[latin1]{inputenc}
\usepackage[normalem]{ulem}
\usepackage{epsfig}
\usepackage{amsfonts}
\usepackage{graphicx}
\usepackage{graphicx,color}
\usepackage{footnote}
\usepackage{amsthm}
\usepackage{amsbsy}
\usepackage{amsmath}
\usepackage{amssymb}
\usepackage{rotating}
\usepackage[colorlinks=true,citecolor=blue,linkcolor=blue]{hyperref}
\usepackage[T1]{fontenc}
\hypersetup{colorlinks=true,linkcolor=blue,filecolor=blue,pagecolor=blue,urlcolor=blue}

\newtheorem{theorem}{Theorem}[section]
\newtheorem{lemma}{Lemma}[section]

\newtheorem{corollary}{Corollary}[section]

\newtheorem{proposition}[theorem]{Proposition}

\theoremstyle{definition}
\newtheorem{definition}{Definition}[section]
\newtheorem{example}{Example}

\theoremstyle{remark}
\newtheorem{remark}{Remark}

\newcommand{\R}{\mathbb{R}}

\newcommand{\s}{\mathbb{S}}

\newcommand{\Z}{\mathbb{Z}}

\newcommand{\Q}{\mathcal{Q}}

\newcommand{\ria}{\rightarrow}

\newcommand{\om}{\omega}

\newcommand{\n}{\nabla}
\newcommand{\ran}{\rangle}
\newcommand{\lan}{\langle}
\newcommand{\ve}{\varepsilon}
\newcommand{\vp}{\varphi}

\DeclareMathOperator{\hess}{Hess}
\DeclareMathOperator{\ric}{Ric}

\DeclareMathOperator{\di}{div}
\DeclareMathOperator{\tr}{tr}

\DeclareMathOperator{\diam}{diam}
\DeclareMathOperator{\dist}{dist}

\DeclareMathOperator{\sect}{Sect}

\makeatletter
\@namedef{subjclassname@2020}{%
  \textup{2020} Mathematics Subject Classification}
\makeatother

\title[Poincar\'e type inequality]
{Poincar\'e type inequality for hypersurfaces and rigidity results}

\author{Hil\'ario Alencar, \, 
M\'arcio Batista\,
\and Greg\'orio Silva Neto\,}

\address{Universidade Fe\-deral de Alagoas, Instituto de Matem\'atica, Macei\'o, 
AL, 57072-900, Brazil}
\email{hilario@mat.ufal.br}
\email{mhbs@mat.ufal.br}
\email{gregorio@im.ufal.br}


\begin{document}

\begin{abstract} 
In this article, under mild constraints on the sectional curvature, we exploit a divergence formula for symmetric endomorphisms to deduce a general Poincar\'e type inequality. We apply such inequality to higher-order mean curvature of hypersurfaces of space forms and Einstein manifolds, to obtain several isoperimetric inequalities, as well as rigidity results for complete $r$-minimal hypersurfaces satisfying a suitable decay of the second fundamental form at infinity. Furthermore, using these techniques, we prove flatness and non-existence results for self-similar solutions to a large class of fully nonlinear curvature flows.

\end{abstract}

\maketitle



\section{Introduction}
 In the last decades, many mathematicians investigated the existence of nice embeddings between spaces of functions and estimates providing regularity of solutions to some PDE's. For a domain $\Omega$ in $\R^n$, a classical estimate that allow us to obtain interesting information on the space $W^{1,p}_0(\Omega)$, $1\leq p<n$, is the Poincar\'e inequality. The reader can learn more about the subject in \cite{dFSZ:19}, \cite{KM:17}, \cite{BGG:17}, \cite{CMP:16}, \cite{KM:13} and references therein. 
 
Various consequences of Poincar\'e type inequalities have been obtained in the literature. For instance, estimates of the volume growth, spectral and regularity of solutions to elliptic equations, estimates of the number of harmonic $L^2$ 1-forms, of the number of ends of a manifold, and others. We point out that some rigidity results are achieved from these inequalities under additional constraints on the curvatures, see \cite{CSZ:97} and \cite{LJ:06}.

\subsection*{Results} In this work, we establish a general Poincar\'e type inequality on submanifolds of suitable Riemannian ambient spaces. Using such estimate and additional mild conditions we obtain rigidity results for hypersurfaces of space forms and of suitable Einstein manifolds, as we briefly describe in the following. 
 \begin{itemize}
     \item [(i)] We prove isoperimetric inequalities for domains of hypersurfaces of $\R^{m+1};$ and
     \item[(ii)] that $(r+1)$-minimal hypersurfaces of the space forms, satisfying a suitable decay on the integral of the $r$-mean curvature over the annuli of geodesic balls at infinity, are foliated by totally geodesic submanifolds, becoming cylinders or totally geodesic hypersurfaces if their Ricci curvature is bounded from bellow.
     \item[(iii)] We also prove that hypersurfaces with a determined constant scalar curvature in Einstein manifolds are totally geodesic, provided the integral of their mean curvature over geodesic spheres satisfy a suitable decay; and
     \item[(iv)] a rigidity result for the hyperplane as the only homothetic self-similar solutions to a wide class of fully nonlinear curvature flows.
 \end{itemize}
 
 \subsection*{Organization of the paper} In section \ref{secnp} we present the basic computations of this work. In section \ref{secpti} we state our main general inequality and apply it for the setting of higher-order mean curvature and to derive isoperimetric inequalities. In section \ref{secrr} we obtain the rigidity results in items (ii) and (iii) above as a consequence of our Poincar\'e type inequality. We conclude the paper by proving, in section \ref{sec-curv-flow}, rigidity results for self-similar solutions to some fully nonlinear curvature flows.


\section{Notations and Preliminaries}\label{secnp}



Let $M$ be a hypersurface of a Riemannian $(m+1)$-manifold $\overline{M}^{m+1}.$ Denote by $\n$ and $\overline{\n}$ the connections of $M$ and $\overline{M}^{m+1},$ respectively. Given $\overline{X}:M\ria T\overline{M}^{m+1}$ a vector field, write $\overline{X}=X^\top + X^\perp,$ where $X^\top\in TM$ and $X^\perp\in TM^\perp.$  Denoting by $\lan\cdot,\cdot \ran$ the metric of $\overline{M}$ and by $B(Y,Z)=\overline{\n}_YZ - \n_YZ$ the second fundamental form of $M,$ where $Y,Z\in TM$ are vector fields, we have
\[
\begin{aligned}
\lan \overline{\n}_Y\overline{X},Z\ran &= \lan \overline{\n}_Y X^\top + \overline{\n}_Y X^\perp,Z\ran\\
&=\lan \overline{\n}_Y X^\top,Z\ran - \lan X^\perp,B(Y,Z)\ran.\\
\end{aligned}
\]
If $\eta$ is a normal vector field, then $X^\perp=\lan \overline{X},\eta\ran\eta.$ It implies
\begin{equation}\label{sec.fund.form}
\begin{aligned}
\lan \overline{\n}_Y\overline{X},Z\ran &= \lan \overline{\n}_Y X^\top,Z\ran - \lan\overline{X},\eta\ran \lan \eta,B(Y,Z)\ran\\
&=\lan \overline{\n}_Y X^\top,Z\ran - \lan\overline{X},\eta\ran\lan A(Y),Z\ran,
\end{aligned}
\end{equation}
where $A:TM\ria TM$ is the Weingarten operator in direction $\eta$ which is given by
\begin{equation}\label{shape}
\lan A(V),W\ran = \lan \eta, B(V,W)\ran, \ V,W\in TM.
\end{equation}



We first state a general divergence formula which will be useful for us in the next section. A similar formula was obtained by the first and third named authors in \cite{AN17}.

\begin{proposition}\label{prop-div-P1-const} If $M$ is a hypersurface of a $(m+1)$-dimensional Riemannian manifold $\overline{M}^{m+1},$ $m\geq2,$ and $\overline{X}:M\ria T\overline{M}$ is a vector field, then, for every symmetric linear operator $T:TM\to TM$ and every smooth function $f:M\to\R,$ it holds
\begin{equation}\label{div-Pr-const}
\begin{aligned}
\di_f(T(X^\top)) &= -\lan X^\top, T(\n f)\ran + \tr\left(E\mapsto T\left(\left(\overline{\n}_E\overline{X}\right)^\top\right)\right)\\
&\quad + \lan \overline{X},\eta\ran \tr(A T) + (\di T)(X^\top).
\end{aligned}
\end{equation}
Here, $\di_f(Y)=e^f\di(e^{-f}Y)$ is the weighted divergence, $(\di T)(Y)=\tr(E\mapsto(\n_E T)(Y)),$ and $\tr$ denotes the trace of the operator.
\end{proposition}
\begin{proof}
Let $\{e_1,e_2,\ldots,e_m\}$ be an orthonormal frame in $TM$ and $\overline{X}\in T\overline{M}.$ First, since $T$ is self-adjoint, we have 
\begin{equation}\label{eq.trXP_1}
\begin{aligned}
\tr\left(E\mapsto T\left(\left(\overline{\n}_E\overline{X}\right)^\top\right)\right) &= \sum_{i=1}^m \left\lan T\left(\left(\overline{\n}_{e_i}\overline{X}\right)^\top\right),e_i\right\ran =\sum_{i=1}^m \lan\overline{\n}_{e_i}\overline{X},T(e_i)\ran.\\
\end{aligned}
\end{equation}
By using (\ref{sec.fund.form}), p.\pageref{sec.fund.form}, and the self-adjointness of $A,$ we obtain
\[
\begin{aligned}
\sum_{i=1}^m \lan \overline{\n}_{e_i}\overline{X}, T(e_i)\ran &= \sum_{i=1}^m \lan \overline{\n}_{e_i} X^\top, T(e_i)\ran - \left(\sum_{i=1}^m \lan A(e_i),T(e_i)\ran\right)\lan \overline{X},\eta\ran\\
&=\sum_{i=1}^m \lan \overline{\n}_{e_i} X^\top, T(e_i)\ran - \left(\sum_{i=1}^m \lan (A T)(e_i),e_i\ran\right)\lan \overline{X},\eta\ran\\
&=\sum_{i=1}^m \lan\overline{\n}_{e_i}X^\top, T(e_i)\ran - \tr(A T)\lan\overline{X},\eta\ran.
\end{aligned}
\]
Thus,
\[
\sum_{i=1}^m \lan\overline{\n}_{e_i}X^\top, T(e_i)\ran = \tr\left(E\mapsto T\left(\left(\overline{\n}_E\overline{X}\right)^\top\right)\right) + \tr(AT)\lan \overline{X},\eta\ran.
\]
On the other hand, the self-adjointness of $T$ implies
\[
\begin{aligned}
\sum_{i=1}^m \lan \overline{\n}_{e_i} X^\top,T(e_i)\ran &= \sum_{i=1}^m \lan \n_{e_i} X^\top + B(e_i,X^\top),T(e_i)\ran \\
&=\sum_{i=1}^m \lan T(\n_{e_i} X^\top), e_i\ran \\
&= \sum_{i=1}^m \lan \n_{e_i}(T(X^\top)),e_i\ran - \sum_{i=1}^m \lan (\n_{e_i} T)(X^\top),e_i\ran\\
&=\di(T(X^\top)) - \tr(E\mapsto (\n_E T)(X^\top))\\
&=\di(T(X^\top)) - (\di T)(X^\top).
\end{aligned}
\]
Therefore, 
\[
\di(T(X^\top)) = \tr\left(E\mapsto T\left(\left(\overline{\n}_E\overline{X}\right)^\top\right)\right) + (\di T)(X^\top) + \tr(A T)\lan \overline{X},\eta\ran.
\]
Since
\[
\di_f(Y)=e^f\di(e^{-f}Y)=\di(Y)-\lan\n f, Y\ran, \ Y\in TM,
\]
we conclude the result.
\end{proof}

In the next lemma we will estimate $\tr\left(E\mapsto T\left(\left(\overline{\n}_E\overline{X}\right)^\top\right)\right)$ for a special vector field $\overline{X},$ in terms of $\tr T,$ the distance function of $\overline{M}$ and the bounds of the sectional curvatures of $\overline{M}.$ This result is essentially in \cite{BM}, Proposition 2.2, p.109, but by the difference of notations between the two articles and by the sake of completeness, we include a (different) proof here.


\begin{lemma}\label{lemma-main1}
Let $\overline{M}^{m+1},$ $m \geq 2,$ be a Riemannian $(m+1)$-dimensional manifold whose sectional curvatures satisfy
\[
\sect_{\overline{M}}(\overline{V},\overline{\n}\rho)\leq -\frac{G''(\rho)}{G(\rho)},\ \forall\ \overline{V}\in T\overline{M},\ \overline{V}\perp\overline{\n}\rho,
\]
for a class $\mathcal{C}^2$ nondecreasing function $G:[0,b)\to\R,$ which is positive on $(0,b)$ for some $b>0,$ and $\rho(x)=\rho(x_0,x)$ is the geodesic distance of $\overline{M}^{m+1}$ starting at a point $x_0\in \overline{M}^{m+1}.$ Let $M$ be a hypersurface of $\overline{M}^{m+1}$ and $T:TM\to TM$ be a nonnegative symmetric linear operator. If $x\in M$ satisfies $\rho(x)<i(\overline{M},x_0),$ where $i(\overline{M},x_0)$ is the injectivity radius of $\overline{M}^{m+1}$ at $x_0,$ then the vector field $\overline{X}=G(\rho)\overline{\n}\rho$ satisfies
\begin{equation}
\tr\left(E\mapsto T\left(\left(\overline{\n}_E\overline{X}\right)^\top\right)\right)(x)\geq G'(\rho(x))(\tr T)(x).
\end{equation}

\end{lemma}

\begin{proof}
Let $\{e_1, e_2,\ldots, e_m\}$ be an orthonormal basis of $T_xM$ composed by eigenvectors of $T$ in $x\in M,$ i.e., 
\[
T(e_i(x))=\theta_i(x)e_i(x), \ i=1,2,\ldots,m.
\]
Since we are assuming that $\rho(x)<i(\overline{M},x_0),$ the function $\rho$ is differentiable. Thus 
\[
\begin{aligned}
\tr\left(E\longmapsto T\left(\left(\overline{\n}_E\overline{X}\right)^\top\right)\right)&= \sum_{i=1}^m \lan\overline{\n}_{e_i}\overline{X}, T(e_i)\ran = \sum_{i=1}^m\theta_i\lan\overline{\n}_{e_i} \overline{X},e_i\ran\\
&=\sum_{i=1}^m\theta_i\lan\overline{\n}_{e_i}\left(G(\rho) \overline{\n}\rho\right),e_i\ran\\
&=\sum_{i=1}^m\theta_i\left[G'(\rho)\lan\overline{\n}\rho,e_i\ran^2 + G(\rho)\hess_{\overline{M}}\rho(e_i,e_i)\right].\\
\end{aligned}
\]
By using the hypothesis and the hessian comparison theorem (see Theorem 2.3, p. 29 of \cite{PRS:2008}), we have
\[
\hess_{\overline{M}}\rho(e_i,e_i)\geq \frac{G'(\rho)}{G(\rho)}[\lan e_i,e_i\ran-\lan\overline{\n}\rho,e_i\ran^2].
\]
This gives
\[
\begin{aligned}
\tr\left(E\mapsto T\left(\left(\overline{\n}_E\overline{X}\right)^\top\right)\right) &\geq G'(\rho)\sum_{i=1}^m \theta_i\lan e_i,e_i\ran = G'(\rho)\sum_{i=1}^m \lan T e_i,e_i\ran\\
&=G'(\rho)(\tr T).
\end{aligned}
\]
\end{proof}

\begin{remark}\label{remark-Rn}
Notice that, in $\overline{M}^{m+1}=[0,b)\times\mathbb{S}^{m},$ with the metric $\lan\cdot,\cdot\ran_{\overline{M}}=dt^2+G(t)^2d\om^2,$ where $d\om^2$ is the metric of $\mathbb{S}^{m},$ the inequality in the Lemma \ref{lemma-main1} becomes an equality and we do not need to assume that $T$ is nonnegative definite, i.e., for hypersurfaces of $\overline{M}^{m+1}$ and for every symmetric linear operator $T:TM\to TM$ we have
\[
\tr\left(E\mapsto T\left(\left(\overline{\n}_E\overline{X}\right)^\top\right)\right)= G'(\rho)(\tr T).
\]
\end{remark}


\section{Poincar\'e type inequality}\label{secpti}

In the next results, we denote by $B_R(x_0)$ the extrinsic ball of $\overline{M}^{m+1}$ with center at $x_0\in\overline{M}^{m+1}$ and radius $R.$ We also denote by $i(\overline{M},x_0)$ the injectivity radius of $\overline{M}^{m+1}$ for geodesics starting at $x_0\in\overline{M}^{m+1}.$

\begin{theorem}\label{Theo-Poincare}
Let $\overline{M}^{m+1},$ $m \geq 2,$ be a Riemannian $(m+1)$-dimensional manifold whose sectional curvatures satisfy
\begin{equation}\label{sectional}
\sect_{\overline{M}}(\overline{V},\overline{\n}\rho)\leq -\frac{G''(\rho)}{G(\rho)},\ \forall\ \overline{V}\in T\overline{M},\ \overline{V}\perp\overline{\n}\rho,    
\end{equation}
for a class $\mathcal{C}^2$ nondecreasing function $G:[0,b)\to\R,$ which is positive on $(0,b)$ for some $b>0,$ and $\rho(x)=\rho(x_0,x)$ is the geodesic distance of $\overline{M}^{m+1}$ starting at a point $x_0\in \overline{M}^{m+1}.$ Let $M$ be a hypersurface of $\overline{M}^{m+1},$ $T:TM\to TM$ be a nonnegative symmetric linear operator and $\Omega\subset M$ be a connected and open domain with compact closure such that $\overline{\Omega}\cap \partial M=\emptyset$. If $\Omega\subset B_R(x_0)$ with $R<i(\overline{M},x_0),$ then, for every class $\mathcal{C}^{1}$ functions $u,f\colon M\ria\R,$ with $u$ nonnegative and compactly supported in $\Omega,$ we have
\begin{equation}\label{Poincare-ineq-geral}
\begin{aligned}
\int_\Omega G'(\rho) u (\tr T) e^{-f}d\mu &\leq G(R)\int_\Omega |T(\n u - u\n f)| e^{-f}d\mu\\
&\ + G\left(R\right)\int_\Omega u \left[||\tr(A  T)| -(\di T)(\n \rho)|\right]e^{-f}d\mu.
\end{aligned}
\end{equation}
Moreover, if $\overline{M}^{m+1}=[0,b)\times\mathbb{S}^{m},$ with the metric $\lan\cdot,\cdot\ran_{\overline{M}}=dt^2+G(t)^2d\om^2,$ where $d\om^2$ is the metric of $\mathbb{S}^{m},$ then it is not necessary to assume that $T$ is nonnegative.
\end{theorem}

\begin{proof}[Proof of Theorem \ref{Theo-Poincare}.]


For every nonnegative class $\mathcal{C}^1$ function $u:M\to\R,$ it holds
\[
\begin{aligned}
\di_f(uT(X^\top))&=e^f\di(e^{-f} u T(X^\top))\\
&=u\di_f(T(X^\top))+\lan\n u, T(X^\perp)\ran,
\end{aligned}
\]
and so we have, using Proposition \ref{prop-div-P1-const} and Lemma \ref{lemma-main1} for $\overline{X}=G(\rho)\overline{\n}\rho,$
\[
\begin{aligned}
\di_f(uT(X^\top))&\geq G(\rho) \lan \n\rho, T(\n u - u\n f)\ran + uG'(\rho)(\tr T)\\
&\quad+ u G(\rho)\lan\overline{\n}\rho,\eta\ran\tr(A  T) + u G(\rho) (\di T)(\n \rho).
\end{aligned}
\]
On the other hand, by divergence theorem, 
\[
\int_\Omega \di_f (u T(X^\top))e^{-f}d\mu = \int_\Omega \di (e^{-f} u T(X^\top))d\mu =0,
\]
which implies, after integration and some rearrangement,
\begin{equation}\label{ineq-poin}
\begin{aligned}
\int_\Omega u G'(\rho)(\tr T) e^{-f}d\mu &\leq \int_\Omega G(\rho) \lan -\n\rho, T(\n u - u\n f)\ran e^{-f}d\mu\\
&\quad+\int_\Omega u G(\rho)\lan-\overline{\n}\rho,\eta\ran\tr(A  T)e^{-f}d\mu\\
&\quad+\int_\Omega u G(\rho)(\di T)(-\n \rho)e^{-f}d\mu.
\end{aligned}
\end{equation}

Since $\Omega\subset B_R(x_0),$ then, for all $x\in\Omega,$ it holds $\rho(x)\leq R.$ Now, since $G$ is increasing and by using Cauchy-Schwartz inequality, we have
\[
\begin{aligned}
\int_\Omega u  G'(\rho)(\tr T)& e^{-f}d\mu \leq G(R)\int_\Omega |T(\n u - u\n f)| e^{-f}d\mu\\
&\quad+G(R)\int_\Omega u \left||\tr(A  T)| - (\di T)(\nabla \rho)\right|e^{-f}d\mu.
\end{aligned}
\]
This gives \eqref{Poincare-ineq-geral}. When $\overline{M}^{m+1}=[0,b)\times\mathbb{S}^{m},$ with the metric $\lan\cdot,\cdot\ran_{\overline{N}}=dt^2+G(t)^2d\om^2,$ the result follows from Remark \ref{remark-Rn}, p.\pageref{remark-Rn}.
\end{proof}

\begin{remark}\label{rem-diam}
 If $\overline{M}^{m+1}$ has constant sectional curvature, then, in the statement of Theorem \ref{Theo-Poincare}, we can choose the base point $x_0$ in order to minimize $R.$ In this case, we can replace $R$ by $(\diam\Omega)/2$ in the Poincar\'e formula \eqref{Poincare-ineq-geral}, assuming $\diam\Omega <2i(\overline{M})$, where $\diam\Omega$ and  $i(\overline{M})$ denote the extrinsic diameter of $\Omega$ and the injectivity radius of $\overline{M}^{m+1}$, respectively.
\end{remark}

\subsection{Space forms and the $r$-mean curvature}

For an oriented hypersurface $M$ of $\overline{M}^{m+1}$, we recall that the eigenvalues $\lambda_1,\lambda_2,\ldots,\lambda_m$ of $A$ are called principal curvatures. The symmetric functions associated to the immersion are given by
\begin{equation}\label{def-Sr}
 S_r = \sum_{i_1<\ldots<i_r}\lambda_{i_1}\cdots\lambda_{i_r},   
\end{equation}

where $(i_1,\ldots,i_r)\in\{1,2,\ldots,m\}^r.$ The $r$-mean curvature of $M$ is defined by
\begin{equation}\label{r-mean}
H_r=\frac{1}{\binom{m}{r}}S_r.    
\end{equation}
When $r=1,$ we have $H_1=H=\frac{1}{m}\tr A,$ the mean curvature of $M$. For $r=2$ and $\overline{M}=\R^{m+1},$ $H_2=\frac{1}{m(m-1)}{\rm Scal},$ where ${\rm Scal}$ is the non-normalized scalar curvature of $M,$ and for $r=m,$ we have that $H_m=\det A$ is the Gauss-Kronecker curvature of $M$. 

We recall that a hypersurface $M$ of $\overline{M}^{m+1}$ is called $r$-minimal if $H_r$ vanishes on $M$. Properties of hypersurfaces involving the $r$-mean curvature, including the case of $r$-minimal hypersurfaces, have been object of research by many authors as, for example, \cite{ENS2011}, \cite{IMR:2011}, \cite{LLS:2011}, \cite{ADR:2013}, \cite{AIR:2013}, \cite{MIR:2015}, \cite{BPR:2018}, and \cite{BPR:2019}.

Associated to the family of higher-order mean curvatures we have the Newton transformations $P_r:TM\to TM$, $r\in\{0,\ldots,m\},$ which are defined recursively as
\[
P_0=I, \ P_r=S_rI-AP_{r-1},
\]
where $I:TM\to TM$ is the identity operator. Clearly $P_r$ is a self-adjoint operator and $AP_r=P_rA.$ This operator has nice properties related with the symmetric functions $S_r.$  We first point out the following properties:

\begin{lemma}\label{properties} For each $0\leq r\leq m-1$ it holds:
\begin{enumerate}
\item $\tr P_r = (m-r)S_r;$
\item $\tr AP_r = (r+1)S_{r+1};$
\item $\tr A^2P_r = S_1S_{r+1} - (r+2)S_{r+2}.$
\end{enumerate}
\end{lemma}
\begin{proof}See \cite{R} and \cite{BC}. \end{proof}

\begin{definition}
Let $\Q^{m+1}_c$ be a $(m+1)$-dimensional, simply-connected, complete Riemannian manifold with constant sectional curvature $c$. If $c>0$ consider $\Q_c^{m+1}=\s_+^{m+1}(c)$ be the open upper hemisphere. We call these manifolds space forms.
\end{definition}

Before stating the consequences of Theorem \ref{Theo-Poincare}, we show sufficient conditions for the divergence of $P_r$ to vanish. Such result is well-known in literature, see \cite{R} and \cite{R:93}. 

\begin{lemma}\label{divpr:01} The divergence of the Newton transformations $P_r$ vanishes, if the ambient manifold  $\overline{M}$ is a space form.
\end{lemma}

In order to state the next Poincar\'e type inequality, we need to define the special functions
\begin{equation}\label{Sc}
\mathcal{S}_c(t)=
\begin{cases}
t,&\mbox{if}\ c=0;\\
\frac{1}{\sqrt{-c}}\sinh(\sqrt{-c}t),&\mbox{if}\ c<0;\\
\frac{1}{\sqrt{c}}\sin(\sqrt{c}t), &\mbox{if} \ c>0.
\end{cases}    
\end{equation}

For space forms and Newton transformations we have the following result:
\begin{theorem}\label{Sr-Rn}
If $M$ is a hypersurface of $\Q_c^{m+1}$ and $\Omega\subset M,$ $\overline{\Omega}\cap \partial M=\emptyset,$ is a connected and open domain with compact closure, then, for every class $\mathcal{C}^{1}$ functions $u,f\colon M\ria\R,$ with $u$ nonnegative and compactly supported in $\Omega,$ we have
\begin{equation}\label{Poincare-ineq-Sr}
\begin{aligned}
\int_\Omega u S_r \mathcal{S}_c'(\rho) e^{-f}d\mu &\leq C_0\int_\Omega \left[|P_r(\n u - u\n f)| +  (r+1)|S_{r+1}|u \right]e^{-f}d\mu,\\
\end{aligned}
\end{equation} 
for $C_0=\frac{1}{(m-r)}\mathcal{S}_c\left(\frac{\diam\Omega}{2}\right)$. In particular, if $P_r:TM\to TM$ is nonnegative definite, then
\begin{equation}\label{Poincare-ineq-Sr-2}
\int_\Omega u H_r  \mathcal{S}_c'(\rho) e^{-f}d\mu \leq C_1\int_\Omega \left[|\n u - u\n f|H_r +  |H_{r+1}|u \right]e^{-f}d\mu,
\end{equation}
for $C_1=(m-r)C_0$. Moreover, the equality holds if $M$ is a geodesic sphere, $\Omega=M,$ and $f,u$ are constant functions.

Here, $\rho:M\to\R_+$ is the distance function of $\Q_c^{m+1},$ restricted to $M,$ from a base point $x_0\in\Q_c^{m+1}$ chosen to minimize the radius of the extrinsic ball $B_R(x_0)\supseteq\Omega,$ (see Remark \ref{rem-diam}), $S_r$ is the $r$-th symmetric function of the eigenvalues of $M,$ $H_r=\binom{m}{r}^{-1}S_r$ is its $r$-mean curvature, and $\diam\Omega$ denotes the extrinsic diameter of $\Omega.$
\end{theorem}

\begin{proof}
Indeed, in $\Q_c^{m+1}$ we have \eqref{sectional} for $G(t)=\mathcal{S}_c(t)$ and using Lemma \ref{divpr:01}, it holds $\di P_r=0$ in space forms. From the second item of Lemma \ref{properties}, we have $\tr(A  P_r)=(r+1)S_{r+1}$ and, by Theorem \ref{Theo-Poincare} and Remark \ref{rem-diam}, we obtain \eqref{Poincare-ineq-Sr}. Moreover, if $P_r$ is nonnegative definite, then\[|P_r(U)|\leq (\tr P_r)|U|=(m-r)S_r|U|,\] which, together with $\binom{m}{r+1}\left(\frac{r+1}{m-r}\right)\binom{m}{r}^{-1}=1,$ gives \eqref{Poincare-ineq-Sr-2}, as desired. In order to verify the equality, just notice that, in geodesic spheres of radius $R$, it holds
\[
\lambda_1=\cdots=\lambda_m=\frac{\mathcal{S}_c'(R)}{\mathcal{S}_c(R)}.
\]
The equality follows by direct substitution.
\end{proof}

\begin{remark}\label{Pr-positive}
There are some conditions to deduce that $P_r$ is nonnegative definite on a connected hypersurface.  We point out some of them below:
\begin{itemize}
\item[(a)] If $S_{r+1}=0,$ then $P_r$ is semi-definite. Thus, if $r$ is odd, then we can choose an orientation such that $P_r$ is nonnegative definite;
\item[(b)] If $S_{r+1}=0,$ $r$ is even, and $S_r\geq 0;$
\item[(c)] If $r$ is odd, $S_{r+1}=0,$ and $S_{r+2}\neq 0$, then we can choose an orientation such that $P_r$ is positive definite;
\item[(d)] If $r$ is even, $S_{r+1}=0,$ $S_{r+2}\neq 0,$ and $S_r\geq 0,$ then $P_r$ is positive definite;
\item[(e)] If $S_k>0$ for some $1\leq k\leq m-1$ and there exists a point where all the principal curvatures are nonnegative, then $P_r$ is positive definite for every $1\leq r\leq k-1.$
\end{itemize}
The proofs of these claims can be found in \cite{C}, Proposition 2.8., p.192, (for items (a) to (d)), \cite{CR}, Proposition 3.2, p.188, (for item (e)).
\end{remark}

In the following, we present some applications of the Poincar\'e inequalities of Theorem \ref{Sr-Rn}. Denote by $d\mu$ the $m$-dimensional Lebesgue measure of $M$ and by $dS_\mu$ the $(m-1)$-dimensional Lebesgue measure of the boundaries of the $m$-dimensional subsets of $M.$ We also denote the volume of a set $\Omega$ by $|\Omega|$ and by $\mathcal{B}_R\subset M$ the geodesic ball of $M$ with radius $R$ and center at a point $x_0\in M,$ and by $\partial\mathcal{B}_R$ its boundary, i.e., the geodesic sphere of radius $R$ and center at $x_0.$ We omit the center of ball in the notation since it will not be important in the statements of the results.
 
\begin{corollary}\label{Area-Sr-Rn}
Let $M$ be a hypersurface of $\R^{m+1}$ such that $H_{r+1}>0,$ $r=1,2,\ldots,m-1,$ and $\Omega\subset M,$ $\Omega\cap\partial M=\emptyset,$ be a connected and open domain with compact closure. If $M$ has a point whose all the principal curvatures are nonnegative, then
\begin{equation}\label{ineq-Sr}
\begin{aligned}
|\Omega|\leq \sum_{k=0}^r\left(\frac{\diam\Omega}{2}\right)^{k+1}\int_{\partial \Omega}H_k dS_\mu +\left(\frac{\diam\Omega}{2}\right)^{r+1}\int_\Omega H_{r+1}d\mu.
\end{aligned}
\end{equation}
Here, $H_r$ is the $r$-mean curvature of $M,$ defined by \eqref{r-mean}, p.\pageref{r-mean}, and $\diam\Omega$ is the extrinsic diameter of $\Omega.$
\end{corollary}
\begin{proof}
Taking $f\equiv 1$ and $u=u_\ve$ in \eqref{Poincare-ineq-Sr-2}, where
\begin{equation}\label{u-eps}
u_\ve(x)=\begin{cases}
1,&\mbox{if} \ \dist(x,\partial\mathcal{B}_R)\geq \ve;\\
\dfrac{1}{\ve}\dist(x,\partial\mathcal{B}_R),& \mbox{if} \ \dist(x,\partial\mathcal{B}_R)< \ve,
\end{cases}
\end{equation}
and $\dist$ stands for the distance function on $M,$ we obtain letting $\epsilon\to 0$ and using the coarea formula,
\begin{equation}\label{vol-Hr}
\int_\Omega H_r d\mu \leq \frac{\diam\Omega}{2}\left[\int_{\partial\Omega} H_rdS_\mu +\int_{\Omega} H_{r+1}d\mu\right].    
\end{equation}
By applying successively \eqref{vol-Hr} and using Remark \ref{Pr-positive}, item (e), we obtain the result.
\end{proof}

\begin{remark}
In particular, for $r=0,$ we have
\begin{equation}\label{isoH}
|\Omega|\leq \left(\frac{\diam\Omega}{2}\right)\left[|\partial\Omega|+\int_{\Omega} H d\mu\right],
\end{equation}
and for $r=1,$ we obtain
\begin{equation}
|\Omega|\leq \left(\frac{\diam\Omega}{2}\right)|\partial\Omega|+\left(\frac{\diam\Omega}{2}\right)^2\left[\int_{\partial \Omega} H dS_\mu + \frac{1}{m(m-1)}\int_\Omega {\rm Scal}\ d\mu\right],
\end{equation}
where $H$ is the mean curvature and ${\rm Scal}$ is the (non-normalized) scalar curvature of $M$. 
\end{remark}
\begin{remark}
Isoperimetric inequalities in the spirit of \eqref{isoH} were obtained by the first and the third authors in \cite{AN18} for immersions in warped product manifolds. We can also compare the previous results with Theorem 2 in \cite{KM:2015}, which states that
\[
\int_M H_k\rho^{p}d\mu \leq \int_M H_r\rho^{p+r-k}d\mu
\]
for every closed hypersurface $M$ of $\R^{n+1}$ satisfying $H_r>0$ and for every $p>0$ and $0\leq k<r$ (compare also with the results of \cite{GR:2020}). Moreover, they prove that equality holds only for round spheres. On its turn, by the proof of our Poincar\'e type inquality \eqref{Poincare-ineq-Sr}, we obtain
\begin{equation}
\int_M H_{r-1}d\mu \leq \int_M\rho|H_r|d\mu
\end{equation}
for every closed hypersurface $M,$ by taking $u$ and $f$ constant functions, with equality holding for round spheres.
\end{remark}

For a weakly locally convex hypersurface (i.e., $M$ has nonnegative second fundamental form), $P_r$ is nonnegative definite for every $r=1,\ldots,m-1.$ Applying successively inequality \eqref{Poincare-ineq-Sr-2}, $m-1$ times, for $f\equiv 1,$ and $u=u_\ve$ we obtain, letting $\ve\to 0$ and using the coarea formula:

\begin{corollary}\label{Area-Sn}
If $M$ is a weakly locally convex hypersurface of $\R^{m+1}$, then the volume of any geodesic ball $\mathcal{B}_R$ of $M$ with radius $R$ satisfies
\begin{equation}\label{ineq-Sn}
\frac{|\mathcal{B}_R|}{R^m}\leq \left[\dfrac{(R\max_{\partial \mathcal{B}_R}|A|)^m-1}{ (R\max_{\partial \mathcal{B}_R}|A|)-1}\right]\frac{|\partial\mathcal{B}_R|}{R^{m-1}} +\int_{\mathcal{B}_R} H_m d\mu,
\end{equation}
where $H_m$ is the Gauss-Kronecker curvature of $M$ and $|A|$ is the matrix norm its second fundamental form. In particular, if there exists $\alpha>0$ such that $\max_{\partial \mathcal{B}_R}|A|\leq \alpha/R,$ then
\begin{equation}\label{ineq-Sn-2}
\frac{|\mathcal{B}_R|}{R^m}\leq C(m,\alpha)\frac{|\partial\mathcal{B}_R|}{R^{m-1}} +\int_{\mathcal{B}_R} H_m d\mu,
\end{equation}
where $C(m,\alpha)=\frac{\alpha^m-1}{\alpha-1}.$ Moreover, if $0<\alpha<1,$ then
\begin{equation}\label{ineq-Sn-3}
\frac{|\mathcal{B}_R|}{R^m}\leq \frac{1}{1-\alpha}\frac{|\partial\mathcal{B}_R|}{R^{m-1}}.
\end{equation}
\end{corollary}
\begin{proof}
Since $\lambda_i\leq|\lambda_i|\leq|A|$ we have $H_r\leq|A|^r.$ Applying \eqref{ineq-Sr} to $\Omega=\mathcal{B}_R$ and $k=m-1,$ we obtain
\[
\begin{aligned}
|\mathcal{B}_R|&\leq \sum_{r=0}^{m-1}R^{r+1}\int_{\partial\mathcal{B}_R}H_rdS_\mu + R^m\int_{\mathcal{B}_R}H_md\mu\\
&\leq\sum_{r=0}^{m-1}R^{r+1}\max_{\partial\mathcal{B}_R}|A|^r|\partial\mathcal{B}_R| + R^m\int_{\mathcal{B}_R}H_md\mu.
\end{aligned}
\]
This implies
\[
\frac{|\mathcal{B}_R|}{R^m}\leq \left[\sum_{r=0}^{m-1}(R\max_{\partial\mathcal{B}_R}|A|)^r\right]\frac{|\partial\mathcal{B}_R|}{R^{m-1}} + \int_{\mathcal{B}_R}H_md\mu,
\]
which gives \eqref{ineq-Sn}. Inequality \eqref{ineq-Sn-2} is an immediate consequence of \eqref{ineq-Sn} and the hypothesis $\max_{\partial \mathcal{B}_R}|A|\leq \alpha/R.$ To conclude \eqref{ineq-Sn-3}, just observe that $H_m\leq |A|^m\leq \alpha^m/R^m,$ which implies
\[
\frac{|\mathcal{B}_R|}{R^m}\leq \frac{\alpha^m-1}{\alpha-1}\frac{\partial\mathcal{B}_R}{R^{m-1}}+\frac{\alpha^m}{R^m}|\mathcal{B}_R|,
\]
which gives the result.
\end{proof}

\begin{remark}\label{rem-Pr-Ar}
In fact, Corollary \ref{Area-Sn} holds for any hypersurface without any assumption of convexity, by applying successively inequality \eqref{Poincare-ineq-Sr}. In this case, \eqref{ineq-Sn} becomes
\begin{equation}\label{ineq-Sn-4}
\frac{|\mathcal{B}_R|}{R^m}\leq \mathcal{C}(m)\left[\dfrac{(R\max_{\partial \mathcal{B}_R}|A|)^m-1}{ (R\max_{\partial \mathcal{B}_R}|A|)-1}\right]\frac{|\partial\mathcal{B}_R|}{R^{m-1}} +\int_{\mathcal{B}_R} |H_m|d\mu,
\end{equation}
where $\mathcal{C}(m)$ is a constant, depending only on $m.$ This constant exists and it holds $\mathcal{C}(m)\leq \frac{2^m-1}{m}.$ In fact, since
\[
P_r = \sum_{k=0}^r(-1)^k S_{r-k}A^k,
\]
and $|S_k|\leq \binom{m}{k}|A|^k,$ we obtain
\[
\begin{aligned}
|P_r|&\leq \sum_{k=0}^r |S_{r-k}||A|^k \leq \left[\sum_{k=0}^r \binom{m}{r-k}\right]|A|^r\\
&=\left[\sum_{k=0}^r \binom{m}{k}\right]|A|^r\leq \left[\sum_{k=0}^{m-1} \binom{m}{k}\right]|A|^r\\
&=(2^{m}-1)|A|^r.
\end{aligned}
\]
By \eqref{Poincare-ineq-Sr} and reasoning as in the proof of Corollary \ref{Area-Sr-Rn}, we obtain
\begin{equation}
|\Omega|\leq \sum_{r=0}^{m-1}\left(\frac{\diam\Omega}{2}\right)^{r+1}\int_{\partial\Omega}\left[\frac{|P_r|}{\binom{m}{r}(m-r)}\right]dS_\mu +\left(\frac{\diam\Omega}{2}\right)^{m}\int_\Omega|H_m|d\mu,
\end{equation}
which gives, for $\Omega=\mathcal{B}_R,$
\[
\begin{aligned}
\frac{|\mathcal{B}_R|}{R^m}&\leq \left[\sum_{r=0}^{m-1}\frac{2^m-1}{\binom{m}{r}(m-r)}(R\max_{\partial\mathcal{B}_R}|A|)^r\right]\frac{|\partial\mathcal{B}_R|}{R^{m-1}}+\int_{\mathcal{B}_R}|H_m|d\mu\\
&\leq\frac{2^m-1}{m}\left[\sum_{r=0}^{m-1}(R\max_{\partial\mathcal{B}_R}|A|)^r\right]\frac{|\partial\mathcal{B}_R|}{R^{m-1}}+\int_{\mathcal{B}_R}|H_m|d\mu,\\
\end{aligned}
\]
since $\binom{m}{r}(m-r)=m\binom{m-1}{r}\geq m.$
\end{remark}

\subsection{Einstein manifolds} Recall that a Riemannian manifold $\overline{M}$ is Einstein if there is a real number $\lambda,$ called Einstein constant, such that its Ricci tensor satisfies
$$
\overline{{\rm Ric}}(X,Y) = \lambda \lan X,Y\ran, X,Y\in T\overline{M}.
$$
Such manifolds are interesting from both mathematical and physical viewpoints. From the wiewpoint of physics, the metric of Einstein manifolds are solutions to the vacuum Einstein field equations. From the mathematical viewpoint, because the metric in such manifolds is a critical point of the total scalar curvature with constraints, see for instance \cite{B:78} for more details.

\begin{example}[Space forms]
The space forms $\mathcal{Q}_c^{m+1}$ are examples of Einstein manifolds with Einstein constant $\lambda=mc.$
\end{example}

Next, we present spaces whose sectional curvature is not constant.

\begin{example}[Product spaces]\label{torus}
Let $\overline{M}=\mathcal{Q}^{p_1}_{c_1}\times\mathcal{Q}^{p_2}_{c_2}$ be the product of two space forms. If $X,Y\in T\mathcal{Q}^{p_1}_{c_1}$ and $V,W\in T\mathcal{Q}^{p_2}_{c_2},$ then the sectional curvatures of $\overline{M}$ are
\[
\sect_{\overline{M}}(X,Y)=c_1, \ \sect_{\overline{M}}(V,W)=c_2, \ \sect_{\overline{M}}(X,V)=0.
\]
This gives $\ric(X)=(p_1-1)c_1$ and $\ric(V)=(p_2-1)c_2.$ Thus, $\overline{M}$ is Einstein if and only if $(p_1-1)c_1=(p_2-1)c_2.$ The same reasoning holds for an arbitrary product $\mathcal{Q}^{p_1}_{c_1}\times\mathcal{Q}^{p_2}_{c_2}\times\cdots\times\mathcal{Q}^{p_k}_{c_k}$ or for an arbitrary product of Einstein manifolds.
\end{example}

\begin{example}[Complex projective space]\label{CP}
The complex projective space $\mathbb{C}P^{m+1}$ is a compact Einstein manifold with sectional curvatures lying in the interval $[1/4,1]$ and Einstein constant $m+2.$
\end{example}

\begin{example}[Schwarzschild metric]\label{SS}
Consider $\mathcal S = \mathbb R^2\times \mathbb S^2$ with the metric 
\begin{equation}\label{schwar:01}
ds^2 = dr^2+ \varphi^2(r)ds_1^2+ \psi^2(r)ds_2^2,
\end{equation}
where we use polar coordinates in the plane $\mathbb R^2$, and $ds_1^2$ and $ds_2^2$ are the metrics on $\mathbb S^1$ and $\mathbb S^2$, respectively. It can be shown that the sectional curvatures of $\mathcal{S}$ satisfy
\[
\sect_{\mathcal{S}}(X,\partial_r)=-\frac{\vp''(r)}{\vp(r)}, X\in T\s^1, \ \sect_{\mathcal{S}}(V,\partial_r)=-\frac{\psi''(r)}{\psi(r)}, V\in T\s^2,
\]
where $-\vp''/\vp=2\psi''/\psi.$ Choose the functions $\varphi$ and $\psi$ verifying the following differential equations:
\begin{equation*}\label{schwar:02}
\begin{cases}
(\psi')^2&= 1+ C\psi^{-1},\\
2\psi''&=-C\psi^{-2},\\
\psi'&= \alpha\varphi,\\
\end{cases}
\end{equation*}
for $\alpha$ and $C$ determined by the initial data. To obtain smoothness of the metric at the origin, we require that $\varphi(0)=0$, $\varphi'(0)=1$ and $\psi(0)=\beta$, for some $\beta>0$. A simple computation gives $C=-\beta$ and $2\alpha=\beta^{-1}$. With this condition, we have $\psi''=(\beta/2)\psi^{-2}>0.$ A straightforward computation show that the family of metrics \eqref{schwar:01} have Ricci curvature zero and so $(\mathcal S, ds^2)$ are Einstein manifolds, for more details see, for instance, \cite{Petersen}, pp.75--76.
\end{example}

We now bring our attention to a family of Einstein manifolds with a warped product metric. Such manifolds are interesting and the reader can learn more about them in \cite{C:00}, \cite{KK:03}, \cite{CSW:11} and \cite{HPW:12}.

\begin{example}(Warped produtcs)
Recall that given two Riemannian manifolds $(M^n, g_M)$ and $(F^m, g_F)$ and a positive smooth function $w$ on $M$, the warped product metric on $M\times F$ is defined by
$$g=g_M+w^2g_F.$$
We denote it as $M\times_wF$. In \cite{CSW:11} the authors notice that $M\times_wF$ is an Einstein manifold if and only if 
$${\rm Ric}_M - \frac{m}{w}{\rm Hess}\, w = \lambda g_M,$$
where $F^m$ is an $m$-dimensional Einstein manifold. If $M$ has nonempty boundary, we assume that $w=0$ on $\partial M$, see \cite{HPW:12}.
\end{example}

If $M$ is a hypersurface of an Einstein manifold, then the first Newton transformation $P_1$ has divergence zero. This fact was proved in \cite{E:02}. 

\begin{lemma}\label{divpr:03} The divergence of the first Newton transformation $P_1$ vanishes if the ambient manifold $\overline{M}$ is Einstein.
\end{lemma}

Notice that, tracing the Gauss equation twice, for an adapted orthonormal frame $\{e_1,e_2,\ldots,e_m,\eta\}$, we have
\begin{equation}\label{traceGauss}
\sum_{i=1}^m\overline{\rm Ric}(e_i,e_i) - \overline{\rm Ric}(\eta,\eta)={\rm Scal} -2S_2.
\end{equation}
In particular, if $\overline{M}$ is an Einstein manifold with Einstein constant $\lambda,$ i.e., $\overline{\rm Ric}(X,Y)=\lambda\lan X,Y\ran,$ then
\begin{equation}\label{eq.scal}
(m-1)\lambda={\rm Scal} -2S_2,
\end{equation}
where ${\rm Scal}$ denotes the scalar curvature of $M.$ 

Fix $x_0\in\overline{M}^{m+1}.$ Let $B_t(x_0)$ be a ball of $\overline{M}^{m+1}$ with center at $x_0$ and radius $t>0,$ and $\overline{\gamma}$ be a geodesic ray such that $\overline{\gamma}(0)=x_0.$ Define $F_{x_0}:[0,i(\overline{M},x_0))\to\R_+$ by
\begin{equation}\label{Fx0}
F_{x_0}(t)=\max_{B_t(x_0)}\left\{\sect_{\overline{M}}(\overline{V},\overline{\gamma}'), \forall\ \overline{V}\in T\overline{M},\overline{V}\perp\overline{\gamma}'\right\},    
\end{equation}
i.e., $F_{x_0}(t)$ is maximum of all the radial sectional curvatures of $\overline{M}^{m+1}$ in the geodesic ball $B_t(x_0).$ Here, $i(\overline{M},x_0)$ is the injectivity radius of $\overline{M}$ at $x_0.$ Since $F_{x_0}$ is a nondecreasing function of $t,$ it is continuous and differentiable almost everywhere. Let $G_{x_0}:[0,i(\overline{M},x_0))\to\R_+$ be a weak solution to
\begin{equation}\label{GF}
G_{x_0}''(t)+F_{x_0}(t)G_{x_0}(t)\leq 0.
\end{equation}
 
\begin{example}
For the examples we presented earlier, we have:
\begin{itemize}
    \item[(i)] $F_{x_0}(t)=c:=\max\{c_1,c_2\}$ for $\overline{M}=\mathcal{Q}_{c_1}^{p_1}\times\mathcal{Q}_{c_2}^{p_2},$ which gives $G_{x_0}(t)=\mathcal{S}_c(t),$ for every $x_0\in\overline{M};$
    \item[(ii)] $F_{x_0}(t)=1$ for $\overline{M}=\mathbb{C}P^{m+1},$ which gives $G_{x_0}(t)=\mathcal{S}_1(t),$ for every $x_0\in\overline{M};$
    \item[(iii)] $F_{x_0}(t)=2\psi''(t)/\psi(t)=-\vp''(t)/\vp(t)$ for $\overline{M}=\R^2\times\s^2$ with the Schwarzschild metric, since $\psi''>0,$ which gives $G_{x_0}(t)=\vp(t)$. Here, $x_0\in\overline{M}$ is the pole, i.e., the point of $\overline{M}$ where $t=0.$
\end{itemize}
Notice that, in all these examples, the function $G_{x_0}$ is nonnegative and nondecreasing. 
\end{example}

For hypersurfaces of Einstein manifolds we have the following Poincar\'e type inequality:
\begin{theorem}\label{Einstein}
Let $\overline{M}^{m+1}$ be an Einstein manifold, with Einstein constant $\lambda.$ Assume there exists $x_0\in\overline{M}^{m+1}$ such that $G_{x_0}$, defined by \eqref{GF}, is nonnegative and nondecreasing in a ball $B_R(x_0),$ and let $\rho(x)=\rho(x,x_0),$ be the distance function of $\overline{M}^{m+1}$ starting at $x_0$. If $M$ is a hypersurface of $\overline{M}^{m+1}$ and $\Omega\subset M\cap B_R(x_0)$ is a connected and open domain, with compact closure, such that $\overline{\Omega}\cap \partial M=\emptyset$ and $R<i(\overline{M},x_0),$ then, for every class $\mathcal{C}^{1}$ functions $u,f\colon M\ria\R,$ with $u$ nonnegative and compactly supported in $\Omega,$ we have
\begin{equation}\label{Poincare-ineq-S12}
\begin{aligned}
\int_\Omega u S_1 G_{x_0}'(\rho) e^{-f}d\mu &\leq C_0\int_\Omega \left[|P_1(\n u - u\n f)| +  |{\rm Scal} -(m-1)\lambda|u \right]e^{-f}d\mu,\\
\end{aligned}
\end{equation} 
where $C_0=\frac{1}{(m-1)}G_{x_0}(R)$ and $G_{x_0}$ is a solution to \eqref{GF}. In particular, if $P_1:TM\to TM$ is nonnegative definite, then
\begin{equation}\label{Poincare-ineq-Sr-12}
\int_\Omega u S_1  G_{x_0}'(\rho) e^{-f}d\mu \leq C_1\int_\Omega \left[|\n u - u\n f|S_1 +  \left|\frac{{\rm Scal}}{m-1}-\lambda\right|u \right]e^{-f}d\mu,
\end{equation}
for $C_1=(m-1)C_0$. 
\end{theorem}
\begin{proof}
Indeed, using Lemma \ref{divpr:03} we have that $\di P_1=0$ on an Einstein manifold. By the definition of $G_{x_0},$ we have that
\[
\sect_{\overline{M}}(\overline{V},\ \overline{\gamma}')\leq-\frac{G_{x_0}''(t)}{G_{x_0}(t)}, \forall\ \overline{V}\in T\overline{M},\ \mbox{with}\ \overline{V}\perp\overline{\gamma}'.
\]
From the second item of Lemma \ref{properties} we have $\tr(A  P_1)=2S_{2}$. Using \eqref{eq.scal} and Theorem \ref{Theo-Poincare}, we obtain \eqref{Poincare-ineq-S12}. Moreover, if $P_1$ is nonnegative definite, then\[|P_1(U)|\leq (\tr P_1)|U|=(m-1)S_1|U|,\] which gives \eqref{Poincare-ineq-Sr-12}, as desired. \end{proof}




\section{Rigidity results}\label{secrr}
In this section we state some rigidity results which are consequences of our Poincar\'e type inequality. Let
\begin{equation}\label{hc}
h_c(t)=
\begin{cases}
t,&\mbox{if}\ c=0;\\
\frac{1}{\sqrt{-c}}\sinh(\sqrt{-c}t),&\mbox{if}\ c<0;\\
1,&\mbox{if} \ c>0.
\end{cases}
\end{equation}
Notice that $\mathcal{S}_c(t)=h_c(t)$ for $c\leq 0$ and $\mathcal{S}_c(t)\leq \sqrt{c}h_c(t)$ for $c>0.$ Recall that $\mathcal{B}_R$ denotes the geodesic ball of $M$ with radius $R$ and center at some point $p_0\in M$. If $M$ is complete and we make $R\to\infty,$ the center $p_0$ of the ball does not matter, and for this reason we omit the center of the ball in the notation of the results of this and the next sections. 

The first result reads as follows:
\begin{theorem}\label{r-mean-zero}
Let $M$ be a complete $(r+1)$-minimal hypersurface, $1 \leq r \leq m-1,$ of a space form $\Q_c^{m+1}$ of constant sectional curvature $c\in\R$ such that $r$ is odd, or $r$ is even and $H_r\geq 0.$ If
\begin{equation}\label{hyp-r-mean-zero}
\liminf_{R\to\infty} \frac{h_c(R)}{R}\int_{\mathcal{B}_R\setminus\mathcal{B}_{R/2}} H_r d\mu =0,
\end{equation}
then $M$ is foliated by $(m-r+1)$-dimensional totally geodesic submanifolds of $\Q_c^{m+1}.$  Moreover,
\begin{itemize}
\item[(i)] if $\Q_c^{m+1}=\R^{m+1}$ and $M$ has nonnegative Ricci curvature, then $M=N^{r-1}\times \R^{m-r+1},$ where $N^{r-1}$ is a $(r-1)$-dimensional Riemannian manifold;
\item[(ii)] if $\Q_c^{m+1}=\s_+^{m+1}(c),$ the open upper hemisphere, and $M$ has Ricci curvature bounded from below by $c,$ then $M$ is totally geodesic.
\end{itemize}
Here, $h_c$ is defined by \eqref{hc}, $H_r$ is the r-mean curvature defined by \eqref{r-mean}, and $\mathcal{B}_R$ is the geodesic ball of $M.$
 \end{theorem}
\begin{proof}
First notice that, by the first item in Remark \ref{Pr-positive}, since $S_{r+1}\equiv 0,$ we have that $P_r$ is semi-definite. If $r$ is odd, we can choose an orientation such that $P_r$ is positive semi-definite. This implies $\frac{1}{m-r}\tr P_r =S_r\geq 0.$ If $r$ is even it does not happen, but the assumption that $S_r\geq 0$ assures that $P_r$ is positive semi-definite. This implies that 
\[
|P_r(U)|\leq (\tr P_r)|U|=(m-r)S_r|U|, \ U\in TM.
\] 
Taking $\Omega=\mathcal{B}_R$ in the inequality \eqref{Poincare-ineq-Sr-2}, we have that $\diam \mathcal{B}_R\leq 2R,$ since the extrinsic distance is less than or equal to the intrinsic distance, and
\begin{equation}\label{integral-teo41}
\int_{\mathcal{B}_R} u S_r\mathcal{S}_c'(\rho) d\mu  \leq \mathcal{S}_c(R)\int_{\mathcal{B}_R} |\n u|S_r d\mu,   
\end{equation}
where $\mathcal{S}_c$ is defined by \eqref{Sc}. Taking a positive cut-off function $u:M\to\R$ such that 
\begin{equation}\label{cutoff}
\begin{cases}
    u\equiv1&\mbox{in}\quad \mathcal{B}_{R/2};\\
    |\n u|\leq C/R&\mbox{in}\quad \mathcal{B}_{R}\setminus\mathcal{B}_{R/2};\\
    u\equiv0&\mbox{in}\quad M\setminus\mathcal{B}_{R},
\end{cases}    
\end{equation}
for some $C>0,$ we obtain
\[
\begin{aligned}
\int_{\mathcal{B}_{R/2}} S_r \mathcal{S}'_c(\rho)d\mu &\leq \int_{\mathcal{B}_R} u S_r \mathcal{S}'_c(\rho)d\mu\\
&\leq \mathcal{S}_c(R)\int_{\mathcal{B}_R} S_r|\n u| d\mu\\
&\leq C\frac{\mathcal{S}_c(R)}{R}\int_{\mathcal{B}_R\setminus\mathcal{B}_{R/2}} S_r d\mu\\
&\leq C\max\{1,\sqrt{c}\}\frac{h_c(R)}{R}\int_{\mathcal{B}_R\setminus\mathcal{B}_{R/2}} S_r d\mu.
\end{aligned}
\] 
Making $R\to\infty,$ we obtain
\[
\int_M S_r \mathcal{S}'_c(\rho)d\mu \leq C\max\{1,\sqrt{c}\}\liminf_{R\to\infty}\frac{h_c(R)}{R}\int_{\mathcal{B}_R\setminus\mathcal{B}_{R/2}} S_r d\mu=0,
\]
which implies that $S_r\equiv 0.$ Since $S_{r+1}\equiv 0\equiv S_r,$ by Lemma 2.1, p.252, of \cite{HL}, we obtain that $A$ has rank at most $r-1,$ i.e., $M$ has index of relative nullity at least $m-r+1.$ By using Proposition 1.18, p.24 of \cite{D}, we conclude that $M$ is foliated by $(m-r+1)$-dimensional totally geodesic submanifolds of $\Q_c^{m+1}.$ If $\Q_c^{m+1}=\R^{m+1}$ and $M$ has nonnegative Ricci curvature, then by using Hartman's splitting theorem (see \cite{D}, Theorem 7.15, p.196), $M=N^{r-1}\times\R^{m-r+1}.$ If $\Q_c^{m+1}=\s_+^{m+1}(c)$ and the Ricci curvature of $M$ is bounded from below by $c,$ then by Corollary 7.12, of \cite{D}, $M$ is totally geodesic.
\end{proof}

\begin{remark}
If we replace the decay condition on $H_r$ in the hypothesis of Theorem \ref{r-mean-zero} by
\[
\liminf_{R\to\infty}\frac{h_c(R)}{R}\int_{ \mathcal{B}_R\setminus\mathcal{B}_{R/2}}|A|^rd\mu =0,
\]
then we do not need assume that $H_r\geq 0$ for $r$ even. Indeed, in this case we can use \eqref{Poincare-ineq-Sr} and the discussion in Remark \ref{rem-Pr-Ar}, to estimate $|P_r|.$
\end{remark}


Let $\overline{M}^{m+1}$ be an Einstein manifold and $M$ be a complete hypersurface of $\overline{M}^{m+1}$. Define, for each $x_0\in\overline{M}^{m+1}$,
\[
\mathcal{G}_{x_0}(t)=
\begin{cases}
G_{x_0}(t),&\mbox{if}\ i(\overline{M},x_0)=\infty\ \mbox{and}\ M\cap (\overline{M}\setminus B_R(x_0))\neq\emptyset, \forall R>0;\\
1,&\mbox{if} \ M\subset B_{R_0}(x_0)\ \mbox{for some}\ R_0>0,\\
\end{cases}
\]
where $G_{x_0}$ is the solution to \eqref{GF}. Thus, for Einstein manifolds, we have:
\begin{theorem}\label{cor-scal}
Let $\overline{M}^{m+1}$ be an Einstein manifold, with Einstein constant $\lambda.$ Assume there exists $x_0\in\overline{M}^{m+1}$ such that $G_{x_0},$ defined by \eqref{GF}, is nonnegative and nondecreasing.  If $M$ is a complete hypersurface of $\overline{M}^{m+1},$ with constant scalar curvature $(m-1)\lambda,$ such that
\begin{equation}\label{hyp-Einstein-zero}
\liminf_{R\to\infty} \frac{\mathcal{G}_{x_0}(R)}{R}\int_{\mathcal{B}_R\setminus\mathcal{B}_{R/2}} H d\mu =0,
\end{equation}
then $M$ is totally geodesic. 

Here, $H$ denotes the mean curvature of $M$ and $\mathcal{B}_R$ denotes the geodesic ball of $M.$ 
\end{theorem}

\begin{proof}
The eigenvalues of $P_1$ are $S_1-\lambda_i,$ where $\lambda_i$ are the principal curvatures of $M.$ Since $S_2\equiv 0$ and $S_1\geq0,$ we have
\[
S_1-\lambda_i \leq S_1 + |\lambda_i|\leq S_1 + \sqrt{\lambda_1^2+\cdots+\lambda_m^2}=S_1 + |A| = 2S_1,
\]
i.e., $|P_1|\leq 2S_1.$ On the other hand, by the definitions of $F_{x_0}$ (see \eqref{Fx0}) and $G_{x_0}$, we have 
\[
\sect_{\overline{M}}(\overline{V},\ \overline{\gamma}')\leq -\frac{G_{x_0}''(t)}{G_{x_0}(t)}, \forall\ \overline{V}\in T\overline{M},\ \mbox{with}\ \overline{V}\perp\overline{\gamma}'.
\]
Following the same reasoning of the proof of Theorem \ref{r-mean-zero}, but applying Theorem \ref{Einstein}, for the cut-off function \eqref{cutoff}, we have
\[
\begin{aligned}
\int_{\mathcal{B}_{R/2}}S_1 G'_{x_0}(\rho)d\mu &\leq \frac{G_{x_0}(R)}{m-1}\int_{\mathcal{B}_R}|P_1(\n u)|d\mu\\ 
&\leq \frac{2m G_{x_0}(R)}{m-1}\int_{\mathcal{B}_R}|\n u|H\mu\\
&\leq \frac{2m}{m-1}\frac{G_{x_0}(R)}{R}\int_{\mathcal{B}_R\setminus\mathcal{B}_{R/2}} Hd\mu.
\end{aligned}
\]
Taking $R\to\infty$ and observing that $G_{x_0}(R)<G_{x_0}(R_0)<\infty$ over $M,$ if $M\subset B_{R_0}(x_0)$ for some $R_0>0,$ we conclude, by using the hypothesis \eqref{hyp-Einstein-zero}, that $S_1\equiv 0.$ This gives $|A|=\sqrt{S_1^2-2S_2}=0,$ i.e., $M$ is totally geodesic.
\end{proof}

Since space forms are particular cases of Einstein manifolds for $\lambda=mc$, we have

\begin{corollary}
If $M$ is a complete hypersurface with constant scalar curvature $m(m-1)c$ in a space form $\Q_c^{m+1}$ of constant sectional curvature $c\in\R,$ such that 
\[
\liminf_{R\to\infty} \frac{h_c(R)}{R}\int_{\mathcal{B}_R\setminus\mathcal{B}_{R/2}} H d\mu =0,
\]
then $M$ is totally geodesic.

Here, $h_c$ is defined by \eqref{hc}, $H$ is the mean curvature, and $\mathcal{B}_R$ is the geodesic ball of $M$ with radius $R.$
\end{corollary}

As a consequence of the proof of Theorem \ref{r-mean-zero}, we obtain:
\begin{corollary}\label{r-mean-zero-3}
There is no complete $(r+1)$-minimal hypersurface, $1 \leq r \leq m-1,$ in a space form $\Q_c^{m+1}$ of constant sectional curvature $c\leq 0,$ such that 
\begin{itemize}
\item[(i)] either $r$ is odd, or $r$ is even and $H_r\geq 0;$
\item[(ii)] $M$ is contained in a geodesic ball of $\Q_c^{m+1},$ and
\item[(iii)] $\displaystyle{\liminf_{R\to\infty}\frac{1}{R} \int_{\mathcal{B}_R\setminus\mathcal{B}_{R/2}} H_r d\mu =0.}$
\end{itemize}
Here, $H_r$ is the $r$-mean curvature defined by \eqref{r-mean}, and $\mathcal{B}_R$ denotes the geodesic ball of $M$ with radius $R.$
\end{corollary}

\begin{remark} We would like to point out that assumptions about the integral growth for $H_r$ on balls are common in the literature; for example, see \cite{E:02}, \cite{AdCN:16} and references therein. More specifically, we notice that, in \cite{AdCN:16}, Do Carmo, the first, and third authors, obtained a non-existence result for hypersurfaces in $\mathbb{R}^4$ with zero scalar curvature, Gauss-Kronecker curvature $H_m$ bounded away from zero, and polynomial growth of the quantity $\int_{\mathcal{B}_R}H\, d\mu.$ Furthermore, there are in the literature some splitting results for hypersurfaces with constant scalar curvature immersed in some space forms; see for example \cite{CY:77} and \cite{AGR:12}. Finally, we see that our results combines some kind of decay of $\int_{\mathcal{B}_R\setminus\mathcal{B}_{R/2}} H_r \, d\mu$ with geometric constrains to produce results of rigidity or non-existence.
\end{remark}

\section{Rigidity of self-similar solutions to curvature flows}\label{sec-curv-flow}

Let $\psi:M^m\to\R^{m+1}$ be hypersurface. The evolution of $\psi(M)$ by the curvature is smooth a one-parameter family $\Psi:M\times I\to\R^{m+1}$ of immersions $\Psi_t:=\Psi(\cdot,t):M\to\R^{m+1}$ solving the initial value problem
\begin{equation}\label{curv-flow}
\left\{ \begin{array}{lll}
\dfrac{\partial \Psi}{\partial t}(x,t)&=& (S_{r+1}(x,t))^\alpha \eta(x,t),\\
\Psi(x, 0)&=& \psi(x),
\end{array}\right.
\end{equation}
for $\alpha\in \R-\{0\}$ and $S_{r+1}$ is defined by \eqref{def-Sr}. The initial value problem \eqref{curv-flow} is also called a curvature flow. These flows have been studied by many authors in the last four decades, see, for example, \cite{Chow:1985}, \cite{Tso:1985}, \cite{Chow:1987}, \cite{Urbas:1990}, \cite{Urbas:1991}, \cite{And:1994}, \cite{And:2007}, their citations, and references therein. We also quote the recent book \cite{ACGL:2020} for an extensive introduction of these flows.

A homothetic solution to the flow \eqref{curv-flow} is a hypersurface satisfying the equation 
\begin{equation}\label{homo-sol}
S_{r+1}^\alpha = \delta \lan {\psi},\eta\ran, 
\end{equation}
for some nonzero real number $\delta$. A hypersurface satisfying \eqref{homo-sol} evolves by dilations and contractions via the flow. If $\delta>0,$ then the hypersurface evolves by dilation and it is called a self-expander. If $\delta<0,$ then the hypersurface evolves by contraction and its is called a self-shrinker. 
\begin{remark}
Homothetic solutions are examples of self-similar solutions, which are those solutions which evolves by flow without changing their shapes. Other examples are the translating solitons, which evolves translating the initial hypersurface in a fixed direction and those which evolves by a rotation of $\R^{m+1}.$ For more details, see \cite{ACGL:2020}.
\end{remark}

Homothetic solutions to curvature flows have received considerable attention in recent years, see, for example \cite{MC:2011}, \cite{CM:2012}, \cite{DLW:2016}, \cite{BCD:2017},\cite{GLM:2018}, \cite{GK:2018}, \cite{CCF:2021}, \cite{MC:2021}, \cite{ASNZ-PP}, and \cite{ASNZ:2022}.

For homothetic solutions to the curvature flow \eqref{curv-flow} we can state:

\begin{theorem}\label{teo-homo-r}
Let $M$ be a complete homothetic solution to the curvature flow \eqref{curv-flow} in $\R^{m+1},$ $1 \leq r \leq m-1,$ such that 
\begin{itemize}
    \item[(i)] $\alpha=p/q,$ where $p$ and $q$ are odd integers;
    \item[(ii)] or $\alpha>0$ and $S_{r+1}\geq 0;$
    \item[(iii)] or $\alpha<0$ and $S_{r+1}>0.$
\end{itemize}
If $\delta S_r\geq 0$ and 
\begin{equation}\label{hyp-teo-homo}
\liminf_{R\to\infty} \int_{\mathcal{B}_R\backslash\mathcal{B}_{R/2}} |A|^r d\mu =0,
\end{equation}
then 
\begin{itemize}
\item[(i)] $M$ is a hyperplane if $\alpha>0;$

\item[(ii)] there is no such hypersurface if $\alpha<0.$
\end{itemize}
Here, $\mathcal{B}_R$ denotes the geodesic ball of $M$ with radius $R$ and $|A|$ is the norm of the second fundamental form.
\end{theorem}

\begin{proof}
Using \eqref{ineq-poin} and Remark \ref{remark-Rn}, we have that
\[
(m-r)\int_\Omega u S_r d\mu = \int_\Omega \lan-\overline{X},P_r(\n u)\ran d\mu +(r+1)\int_\Omega u\lan-\overline{X},\eta\ran S_{r+1} d\mu,
\]
for $\overline{X}=\rho\overline{\n}\rho.$ Since $S_{r+1}^{\alpha}=\delta\lan\psi,\eta\ran$ and $\psi=\overline{X}$ in $\R^{m+1},$ we have
\[
\begin{aligned}
\int_\Omega u [(m-r) \delta S_r + (r+1)S_{r+1}^{\alpha+1}]d\mu &= \delta \int_\Omega \lan -\rho\n\rho,P_r(\n u)\ran d\mu\\
&\leq \frac{|\delta|\diam\Omega}{2}\int_\Omega |P_r||\n u|d\mu\\
&\leq \frac{(2^m-1)|\delta|\diam\Omega}{2}\int_\Omega |A|^r|\n u|d\mu,\\
\end{aligned}
\]
since $|P_r|\leq (2^m-1)|A|^r$ by Remark \ref{rem-Pr-Ar}. Taking $\Omega=\mathcal{B}_R$ a geodesic ball of $M$ with radius $R,$ we have $\diam \mathcal{B}_R\leq 2R,$ since the extrinsic distance is less than or equal to the intrinsic distance. Using the cut-off function defined in \eqref{cutoff}, we obtain
\[
\int_{\mathcal{B}_{R/2}} u [(m-r)\delta S_r + (r+1)S_{r+1}^{\alpha+1}]d\mu \leq c(m,r)|\delta|\int_{\mathcal{B}_R\setminus\mathcal{B}_{R/2}} |A|^r d\mu.\\
\]
Notice that, if $\alpha=\frac{2a+1}{2b+1},$ $a,b\in\Z,$ then 
\[
S_{r+1}^{\alpha+1}=\left(S_{r+1}^{\frac{a+b+1}{2b+1}}\right)^2\geq 0
\] no matter the signal of $S_{r+1}.$ Taking $R\to\infty$ and using the hypothesis, we obtain 
\[
S_r\equiv S_{r+1}\equiv 0\equiv\lan\psi,\eta\ran,
\]
which implies that $M$ is a hyperplane, since it is smooth, for $\alpha>0$, (see also \cite{DT}) and leads to a contradiction for $\alpha<0,$ since $S_{r+1}^\alpha$ is defined, in this case, only for $S_{r+1}>0$.
\end{proof}

\begin{remark}
The proof of the Theorem \ref{teo-homo-r} holds in the general setting of a Riemannian manifold $\overline{M}^{m+1}$ with bounded sectional curvatures by $G''(t)/G(t).$ So, we consider hypersurfaces satisfying the equation
\[
S_{r+1}^{\alpha}=\delta\lan G(\rho)\overline{\n}\rho,\eta\ran.
\]
These surfaces have been object of research in recent years as self-similar solutions to curvature flows in ambient spaces other than $\R^{m+1},$ see, for example, \cite{ALR:2020} and \cite{CMR:2020}.
If 
\[
\begin{cases}
\displaystyle{\liminf_{R\to\infty} \frac{G(R)}{R}\int_{\mathcal{B}_R\setminus\mathcal{B}_{R/2}}|A|^r d\mu =0},&\mbox{if} \ G \ \mbox{is unbounded};\\
\displaystyle{\liminf_{R\to\infty}\frac{1}{R} \int_{\mathcal{B}_R\setminus\mathcal{B}_{R/2}}|A|^r d\mu=0}, &\mbox{if} \ G \ \mbox{is bounded},\\
\end{cases}
\]
then $M$ satisfies 
\begin{equation}\label{eq-equiv}
S_{r}\equiv S_{r+1}\equiv 0\equiv\lan G(\rho)\overline\nabla\rho,\eta\ran,
\end{equation} 
(assuming $c\neq 0$). The classification of hypersurfaces satisfying \eqref{eq-equiv} depends on the ambient space we are considering. In the space form $\Q_c^{m+1}$, for example, by using the results in \cite{DT} and \eqref{eq-equiv}, we can conclude that $M$ is totally geodesic if $\alpha>0$ and that the hypersurface does not exist if $\alpha<0$.
\end{remark}

\section*{Acknowledgements}
The authors would like to thank the referee for reading the manuscript in great detail and for his valuable suggestions and useful comments, which improved the paper.

\section*{Funding}
The authors were partially supported by the Alagoas Research Foundation (FAPEAL), by the National Council for Scientific and Technological Development (CNPq), and Coordination for the Improvement of Higher Education Personnel (CAPES), Brazil.

\bibliographystyle{acm}
\bibliography{references}

\end{document}